\documentclass[a4paper, 10pt]{extarticle}
\usepackage[top=70pt,bottom=70pt,left=75pt,right=77pt]{geometry}

\usepackage{amssymb, amsmath, amsthm, setspace, bbm, prettyref, graphicx, float, subfigure, color, mathrsfs, mathtools}

\usepackage{Lutsko_Style}
\DeclareMathSizes{10}{9}{7}{5}
\usepackage[normalem]{ulem}
\usepackage{soul}

\usepackage{graphicx}

\global\long\def\uab{u}%

\newcommand{\mmod}[1]{{{\,\,\mathrm{mod}\,\,#1}}}

\title{
Correlations of the Fractional Parts of $\alpha n^\theta$
}
 \author{
{\sc Christopher Lutsko
and Niclas Technau
}
}
\date{}

\begin{document}

\maketitle

  \setlength{\abovedisplayskip}{1mm}

\begin{center}
{\large\sl Dedicated to Zeev Rudnick on his $60^{th}$ birthday. }
\end{center}
\vspace{2mm}
  
  \begin{abstract}
  \noindent

  Let $m\geq 3$, we prove that $(\alpha n^\theta \mmod 1)_{n>0}$ has Poissonian $m$-point correlation for all $\alpha>0$, provided $\theta<\theta_m$, where $\theta_m$ is an explicit bound which goes to $0$ as $m$ increases. This work builds on the method developed in Lutsko-Sourmelidis-Technau (2021), and introduces a new combinatorial argument for higher correlation levels, and new Fourier analytic techniques. A key point is to introduce an `extra' frequency variable to de-correlate the sequence variables and to eventually exploit a repulsion principle for oscillatory integrals. Presently, this is the only positive result showing that the $m$-point correlation is Poissonian for such sequences.  
  
  \end{abstract}

  \section{Introduction}
  In the following, let $m\geq 2$ be an integer, and let $f \in C_c^\infty(\R^{m-1})$ be a compactly supported function which can be thought of as a stand-in for the characteristic function of a Cartesian products of compact intervals in $\R^{m-1}$. Let $\Vert \cdot\Vert$ be the distance to the nearest integer, and $[N]:=\{1,\ldots,N\}$ where $N\geq 1$ is a large parameter which is taken to $\infty$. Given a sequence $(x(n)) = (x(n))_{n>0} \subseteq \R/\Z$  we define its \emph{$m$-point correlation}, at time $N$, to be
  \begin{equation}\label{def: m-point correlation function}
    R^{(m)}(N,f) := \frac{1}{N} 
    \sum_{\vect{n} \in [1,N]^m}^\ast 
    f(N\|x(n_1)-x(n_2)\|,N \|x(n_2)-x(n_3)\|, \dots, N\|x(n_{m-1})-x(n_{m})\|),
  \end{equation}
  where $\ds \sum^\ast$ denotes a sum over distinct $m$-tuples. Thus the $m$-point correlation measures how correlated points are on the scale of the average gap between neighboring points (which is $N^{-1}$). We say $(x(n))$ has \emph{Poissonian $m$-point correlation} if
  \begin{align}\label{def: expectation}
    \lim_{N \to \infty} R^{(m)}(N,f) = \int_{\R^{m-1}} f(\vect{x}) \mathrm{d}\vect{x} =: \expect{f} 
    \,\, \mathrm{for\,any\,} f \in C_c^\infty(\R^{m-1}).
  \end{align}
  That is, if the $m$-point correlation converges to the expected value if the sequence was uniformly distributed on the unit interval. The key object in this paper are the dilated monomial sequences:
\begin{equation}\label{def: monomial}
  x(n):=\alpha n^\theta \mmod 1.
\end{equation}
The following is our main result.
    
\begin{theorem}\label{thm:main}
  For any $m\ge 3$ the sequence $(\alpha n^\theta \Mod 1)_{n>0}$ has Poissonian $m$-point correlation for any $0< \theta <  1/(m^2+m-1)$, and any $\alpha >0$. 
    \end{theorem}
\begin{remark}\label{rmk: previous work}
The authors and Sourmelidis \cite{LutskoSourmelidisTechnau2021}
recently established that \eqref{def: monomial} has Poissonian 
$2$-point correlation for all $\theta < 14/41$ and all $\alpha >0$.
\end{remark}

 As $m$ increases, the range of $\theta$ decreases. This is to be expected since, for example, the sequence $(n^{1/m})_{n >0}$ does not have Poissonian $m$-point correlations since the $m^{th}$ powers accumulate at $0$. The precise range of $\theta$ in Theorem \ref{thm:main} comes from estimates on exponential sums and oscillatory integrals. If we could achieve square root cancellation in the sums which arise, we would be able to prove Theorem \ref{thm:main} for $\theta < 1/m$. Theorem \ref{thm:main} is thus far from optimal. A more careful analysis using these methods could possibly yield an improved range of $\theta$, but not going beyond $\theta < 1/m$ without significant new ideas. This motivates the following conjecture:

\begin{conjecture}\label{conj:conj}
  For any $m\ge 2$ the sequence $(\alpha n^\theta \mod 1)_{n>0}$ has Poissonian $m$-point correlation for any $0 <\theta < 1/m$ and any $\alpha >0$.
\end{conjecture}
   
    Again, we emphasize that the discrepancy between Conjecture \ref{conj:conj} and Theorem \ref{thm:main} is technical in nature and derives from suboptimal exponential sum bounds. 
The only real obstruction for the $m$-point correlation is the sequence 
$( n^{1/m} \mmod 1)_{n>0}$ where the $m^{th}$-powers accumulate at $0$ 
and thus prevent Poissonian correlations. However it should be noted 
that El-Baz, Marklof and Vinogradov have shown that $(\sqrt{n} \mmod 1)_{n>0}$ 
does have Poissonian pair correlation, if one removes all those $n$ which are squares.

While $x_n := log(n) \Mod{1} $ does not have Poissonian gap distribution, our main theorem motivates the idea that a sequence growing faster than $\log(n)$ and slower than any power of $n$ appears to have Poissonian local statistics. We plan to address this question in a forthcoming paper.
\vspace{2mm}

\textbf{Combinatorial Argument:} One of the key steps in our proof is to complete the sums defining the $m$-point correlation, that is to consider
  \begin{equation}\label{def: m-point correlation function}
    \frac{1}{N} 
    \sum_{\vect{n} \in [1,N]^m}
    f(N\|x(n_1)-x(n_2)\|,N \|x(n_2)-x(n_3)\|, \dots, N\|x(n_{m-1})-x(n_{m})\|).
  \end{equation}
  Then, in Section \ref{s:Completion}, using a combinatorial argument, we are able to show that if this sum converges to a specified target, then the $m$-point correlation is indeed Poissonian. Something similar was previously done in \cite{RudnickSarnak1996} for a different distribution. We emphasize that Section \ref{s:Completion} is of independent interest for more general sequences. However the statement relies on a complex combinatorial argument, therefore we do not summarize the results here.

\subsection{History}\label{sec: Short history}

In 1998 Rudnick and Sarnak \cite{RudnickSarnak1998}, showed that the $2$-point correlation (or pair correlation) of \eqref{def: monomial} is Poissonian for any integer $\theta\geq 2$, and (Lebesgue) almost every $\alpha>0$. Two decades later \cite{AistleitnerEl-BazMunsch2021} and \cite{RudnickTechnau2021} proved the same statement for all non-integer $\theta>1$, and  $0<\theta <1$ respectively. However, excluding these metric results, very little is known about sequences on the unit interval growing with polynomial rate.

Proving deterministic results can often be facilitated by arithmetic structure. For example, the renormalized spacings of quadratic residues modulo $q$ have been investigated by Kurlberg and Rudnick \cite{RudnickKurlberg1999} who showed that the appropriate $m$-point correlation functions in this setting are all Poissonian as the number of prime factors of $q$ tends to infinity. We refer to Boca and Zaharescu \cite{BocaZaharescu2000} for a theory of the pair correlation function of quadratic polynomials in finite fields. Moreover there has been some recent work by Kurlberg and Lester on the spacing statistics of lattice points on circles, where again, the arithmetic structure plays an important role \cite{KurlbergLester2021}.

When working on the unit interval, for sequences of the form $x_n = \alpha n^\theta \Mod{1}$, the only explicit result concerning correlations is due to El-Baz, Marklof, and Vinogradov \cite{El-BazMarklofVinogradov2015a} who used the dynamics of theta-sums to show that 
\begin{equation}\label{def: root n sequence}
(\sqrt{n} \mmod 1)_{n\geq 1,\,\mathrm{not\,a\,square}}
\end{equation}
has Poissonian $2$-point correlation. This is somewhat surprising since Elkies and McMullen \cite{ElkiesMcMullen2004} had established, via quantitative non-divergence in the space of lattices, that the gap distribution of \eqref{def: root n sequence} is {\em not} Poissonian.

For $m \ge 3$ there are hardly any results on the probabilistic theory for $m$-point correlation functions and even fewer deterministic results. An exception is the work of Yesha and the second named author \cite{TechnauYesha2020}, who showed that $(n^\alpha \mod 1)_n$ has Poissonian $m$-point correlation,  for almost all $\alpha> 4m^2 -4m -1$. Moreover, for lacunary sequences we refer to Rudnick and Zaharescu \cite{RudnickZaharescu1999,RudnickZaharescu2002}, for dilations of lacunary integer sequences; and Chaubey and Yesha \cite{ChaubeyYesha2021} where this is extended to dilations of real-valued sequences.

Similarly, Rudnick, Sarnak, and Zaharescu \cite{RudnickSarnakZaharescu2001}, and Fassina, Kim, and Zaharescu \cite{FassinaKimZaharescu2021} also studied the the $m$-point correlation functions along lacunary sub-sequences of $N$.

   \subsection{Plan of Paper}
   The proof of Theorem \ref{thm:main} is roughly the same for all values of $m$. First, this will be an inductive argument: assume the sequence has $k$-point correlations for all $k< m$ (note that the range of $\theta$ decreases as $m$ increases). Then we argue in roughly three steps.\\
\\
\underline{Step 1:}
First we relate the problem to the $m^{th}$-moment of a random variable. This will effectively decorrelate the sequence elements, at the cost of introducing a new frequency variable. Then, following the example of \cite{RudnickSarnak1996} we complete the sums to aid the analysis. As a result, we need to do some combinatorial book-keeping of adding and subtracting terms to isolate a 'target' main term. This combinatorial argument, which allows us to complete the sum, is of interest for any sequence. As such Section \ref{s:Completion} is written for a general sequence.
  \\
\\
\underline{Step 2:}
Using various smooth partitions of unity
and approximations to indicator functions, we Fourier expand the counting problem.
This reduces the problem to an asymptotic evaluation of the $L^{m}([0,1])$ -norm of a two dimensional exponential sum. 
We use a variant of van der Corput's $B$-process 
(Poisson summation plus a stationary phase expansion) 
to shorten the ranges of the exponential sums in the $m^{th}$-power. Then we apply the $B$-process a second time in a different variable to maximize the saving. When running the $B$-process some care is needed since we need rather good error terms -- somewhat better than one finds in the classical literature. This forces us to do the $B$-process by hand, and to use a second order, rather than a first order, expansion of the arising oscillatory integrals. If we stop here, then our bound on the error term is of size $O(N^{m\theta})$.
\\ 
\\
\underline{Step 3:}
Next we expand the $L^m([0,1])$-norm,
and estimate oscillatory integrals of the shape
$$
\int_0^1 e(c\sum_{l\leq L} \pm r_l(h_l-s)^{1/\theta}) \mathrm{d}s,
$$ 
where $c\in \C$ is a constant only depending on $\alpha$, and $\theta$,
and $r_l,h_l\approx N^\theta$ for $l\leq L \leq m$. 
The arising main term comes from 
the regime where the phase function $$ \sum_{l\leq L} \pm r_l(h_l-s)^{1/\theta}$$
vanishes identically. 
The remaining terms, due to the polynomial nature of the phase function, 
admit a non-trivial bound. We  show that such a phase function has the property that,
 at any given point $s\in [0,1]$, at least one of the first $m$-derivatives is large and thus we conclude by applying a localised version of van der Corput's lemma, which allows us to bound the error term by $o(1)$ provided $\theta$ is in the given range.

    \vspace{2mm}

\textbf{Notation:}
Throughout, we use the usual Bachmann--Landau notation: for functions $f,g:X \rightarrow \mathbb{R}$, defined on some set $X$, we write $f \ll g$ (or $f=O(g)$) to denote that there exists a constant $C>0$ such that $\vert f(x)\vert \leq C \vert g(x) \vert$ for all $x\in X$. Moreover let $f\asymp g$ denote $f \ll g$ and $g \ll f$, and let $f = o(g)$ denote that $\frac{f(x)}{g(x)} \to 0$.

Given a Schwartz function $f: \R^m \to \R$ let $\wh{f}$ denote the $m$-dimensional Fourier transform:
\begin{align*}
  \wh{f}(\vect{k}) : = \int_{\R^m} f(\vect{x}) e(-\vect{x}\cdot \vect{k}) \mathrm{d}\vect{x}.
\end{align*}
Here, and throughout we let $e(x):= e^{2\pi i x}$.

All of the sums which appear range over integers, in the indicated interval. We will frequently be taking sums over multiple variables, thus if $\vect{u}$ is an $m$-dimensional vector, for brevity, we write
\begin{align*}
  \sum_{\vect{k} \in [f(\vect{u}),g(\vect{u}))}F(\vect{k}) = \sum_{k_1 \in [f(u_1),g(u_1))} \dots \sum_{k_m \in [f(u_m),g(u_m))}F(\vect{k}).
\end{align*}
Moreover, all $L^p$ norms are taken on $[0,1]$ with respect to Lebesgue measure. Let
\begin{align*}
  \Z^\ast:= \Z \setminus \{0\}. 
\end{align*}
      
  As $\alpha, \theta, $ and $f$ are considered fixed, we suppress any dependence in the implied constants. Moreover, for ease of notation, $\varepsilon>0$ may vary from line to line by a bounded constant. Further, we will frequently encounter the exponent 
\begin{align*}
  \Theta := \frac{1}{1-\theta}.
\end{align*}

\section{Preliminaries}
\label{ss:The A and B Processes}

The following stationary phase principle is derived from the work 
of Blomer, Khan and Young \cite[Proposition 8.2]{BlomerKhanYoung2013},
is a key technical device for us.

\begin{lemma}[Stationary Phase Lemma] \label{lem: stationary phase}
  Let $\Phi$ and $\Psi$ be smooth, real valued functions defined on a compact interval $[a, b]$. Let $\Psi(a) = \Psi(b) = 0$. Suppose there exists constants $\Lambda_\Phi,\Omega_\Psi,\Omega_\Phi \geq 3$ so that
  \begin{align} \label{eq: growth conditions}
    \Phi^{(j)}(x) \ll \frac{\Lambda_\Phi}{ \Omega_\Phi^j}, \ \ \Psi^{(j)} (x)
    \ll \frac{1}{\Omega_\Psi^j}\, \ \ \ 
    \text{and} \ \ \ \Phi^{(2)} (x)\gg \frac{\Lambda_\Phi}{ \Omega_\Phi^2}
  \end{align} 
  for all $j=0, \ldots, 4$ and all $x\in [a,b]$. If $\Phi^\prime(x_0)=0$ for a unique $x_0\in[a,b]$,  and if $\Phi^{(2)}(x)>0$, then
  \begin{equation*}
    \begin{split}
      \int_{a}^{b} e(\Phi(x)) \Psi(x) \rd x = &\frac{e(\Phi(x_0) + 1/8)}{\sqrt{\abs{\Phi''(x_0)}}}
      \Psi(x_0)
      + O\left( \frac{\Omega_\Phi} {\Lambda_{\Phi}^{3/2+O(\varepsilon)}}
      \right),
    \end{split}
  \end{equation*}
  provided $\Omega_\Phi /\Omega_\Psi \ll \log \Omega_{\Phi}$. If instead $\Phi^{(2)}(x)<0$ on $[a,b]$ then the same equation holds with $e(1/8)$ replaced by $e(-1/8)$.
\end{lemma}

\noindent Moreover, we also need the following version of van der Corput's lemma (\cite[Ch. VIII, Prop. 2]{Stein1993}).

\begin{lemma}[van der Corput's lemma] \label{lem: van der Corput's lemma}
    Let $[c,d]$ be a compact interval.
    Let $\Phi,\Psi:[c,d]\rightarrow\mathbb{R}$ be smooth functions. 
    Assume $\Phi''$ does not change sign
    on $[c,d]$ and that for some $i\geq 1$ and $\Lambda>0$ the bound 
    \[
    \vert\Phi^{(i)}(x)\vert\geq\Lambda
    \]
    holds for all $x\in [c,d]$.
    Then
    \[
    \int_{c}^{d}\,e(\Phi(x)) \Psi(x)
    \,\mathrm{d}x \ll \Big(\vert \Psi(d) \vert + \int_{c}^d \vert 
    \Psi'(x)\vert\, \mathrm{d}x \Big) 
    \Lambda^{-1/i}
    \]
    where the implied constant depends only on $i$.
\end{lemma}

\section{Combinatorial Completion}
\label{s:Completion}
To begin with, we setup the problem for the triple correlations, as the general setup is rather more complicated. The key insight in both cases is the following: using a well-known trick (see, for example, \cite{Marklof2007} for the pair correlation) one can express the completed $m$-point correlation as the $m^{th}$ moment of a particular random variable. In so doing, we effectively de-correlate the sequence elements, at the cost of introducing a new variable, and the benefit of introducing an oscillatory integral. This de-correlation will prove crucial, as it allows us to apply one-dimensional techniques without accumulating error terms. Since this process has applications to more general sequences, in the current section, let $(y(n))_{n>0}$ be a sequence on $\R_{>0}$ and let $x(n) := y(n) \Mod{1}$. 

Without using this trick, one could hope to apply multi-dimensional stationary phase arguments in the same way. However, the size of the determinant of the Hessian is difficult to understand and one needs to contend with the accumulation of error terms.

\subsection{Setup of the Problem: Triple Correlation}
\label{ss:Setup 3}

Assume the sequence $(x(n))$ has Poissonian pair correlations. To access the triple correlation, it is more convenient to work with the following random variable. Let $f$ be a $C_c^\infty(\R)$ function, and define
\begin{align*}
  S_N(s)=S_N : = \sum_{n \in [N]} \sum_{k \in \Z} f(N(y(n) +k  +s)).
\end{align*}
Note that if $f$ was the indicator function of an interval $I$, then $S_N$ would count the number of points in $(x_n)_{n \le N}$ which land in the shifted interval $I/N+s$. Now consider the third moment of $S_N$. That is (assuming for simplicity $f \ge 0$)
\begin{align}
  \cM^{(3)}(N) &:= \int_0^1 S_N^3(s) \,\mathrm{d}s\notag\\
              &= \int_0^1 \sum_{\vect{n} \in [N]^3}\sum_{\vect{k} \in \Z^3}  f(N(y(n_1) +k_1 + s)) f(N(y(n_2) +k_2 + s)) f(N(y( n_3) +k_3 + s)) \mathrm{d} s.\notag
\end{align} 
Moving the $\vect{n},k_1,k_2$ sum outside of the integral and changing variables $s \mapsto (N^{-1}s - y(n_3))$ yields (for $N$ large enough)
\begin{align}
  \cM^{(3)}(N)
  &= \frac{1}{N} \int_\R \sum_{\vect{n} \in [N]^3}\sum_{\vect{k} \in \Z^3}  f(N(y(n_1)-y(n_3) +k_1 + s)) f(N(y(n_2)-y(n_3) +k_2 + s)) f(Nk_3 + s) \mathrm{d} s \notag\\
&=  \frac{1}{N}\sum_{\vect{n} \in  [N]^3} \sum_{\vect{k}\in \Z^2}\int_\R f(N (y(n_1)-y(n_3) +k_1) + s) f(N( y(n_2)-y(n_3) +k_2) + s) f(s)\mathrm{d} s  \notag\\
              &=  \frac{1}{N}\sum_{\vect{n} \in  [N]^3} \sum_{\vect{k}\in \Z^2}F\left(N(  y(n_1)-y(n_3) +k_1), N( y(n_2)-y(n_3) +k_2)\right) \label{M = trip},
\end{align} 
where $F(x,y) : = \int_\R f(x+s)f(y+s)f(s) \,\mathrm{d} s$.
 That is, by considering the third moment of $S_N$, 
we recover the (completed) triple correlation of $F$.

\emph{If the sequence $x(n)$ had Poissonian triple correlations} then:
\begin{align*}
  \frac{1}{N}\sum^\ast_{\vect{n} \in  [N]^3} \sum_{\vect{k}\in \Z^2}
F(N(y(n_1) -y(n_3) +k_1), N( y(n_2) - y(n_3) +k_2)) 
\rightarrow \int_{\R^2} F(x,y) \, \mathrm{d}x \mathrm{d}y = \expect{f}^3.
\end{align*}
Now, if $ n_1= n_3\neq n_2$, then by inspection of \eqref{M = trip}, 
and the compactness of $f$, we recover the pair correlation of $F(0,x)$, 
which, by the assumption that $(x(n))$ has Poissonian pair correlations, 
converges to $\expect{F(0,x)} = \expect{f}\mathbf{E}(f^2)$. 
Moreover if $n_1=n_2=n_3$, we have the trivial sum $F(0,0) = \mathbf{E}(f^3)$. 
From here, we conclude that, $(x(n))$ has Poissonian triple correlations if and only if
\begin{align}\label{target triple}
  \cM^{(3)}(N) \to \expect{f}^3 + 3 \expect{f} \mathbf{E}(f^2) + \mathbf{E}(f^3),
\end{align}
as $N \to \infty$.

With that target in mind, first we apply Poisson summation to the sums over $n_i$, to see that
\begin{align*}
  \cM^{(3)}(N)  &= \frac{1}{N^3}\int_0^1 \sum_{\vect{n} \in [N]^3}
  \sum_{\vect{k} \in \Z^3} \wh{f}\Big(\frac{k_1}{N}\Big)\wh{f}\Big(\frac{k_2}{N}\Big)\wh{f}\Big(\frac{k_3}{N}\Big)    
  e( k_1y(n_1)+ k_2 y(n_2) + k_3y(n_3)+ (k_1+k_2+k_3)s) 
  \,\mathrm{d} s.
\end{align*} 
Now suppose $k_3 = 0$, then we obtain
\begin{align*}
  \expect{f}\frac{1}{N^2}\int_0^1 \sum_{\vect{n} \in [N]^2}
  \sum_{\vect{k} \in \Z^2} 
  \wh{f}\Big(\frac{k_1}{N}\Big)\wh{f}\Big(\frac{k_2}{N}\Big)    
  e( k_1y(n_1)  + k_2y(n_2) + (k_1+k_2)s) 
  \,\mathrm{d} s,
\end{align*} 
which is exactly $\expect{f}$ times the second moment of $S_N$. Therefore, this converges to 
$\expect{f}\mathbf{E}(f^2) + \expect{f}^3$. Thus, by symmetry
\begin{align*}
  \cM^{(3)}(N)  &= \cE(N) + \cP(N) + o(1)
\end{align*} 
where $\cP(N) \to 3\expect{f}\mathbf{E}(f^2)+ \expect{f}^3$ as $N \to \infty$ (the term $\expect{f}^3$ comes from $k_1=k_2=k_3=0$, and is thus only counted once), and where
\begin{align*}
  \cE(N) : = \frac{1}{N^3}\int_0^1 \sum_{\vect{n} \in [N]^3}
\sum_{\vect{k}\in (\Z^\ast)^3 } 
\wh{f}\Big(\frac{\vect{k}}{N}\Big)   
e(\vect{k}\cdot y(\vect{n})  +\vect{k} \cdot \vect{1}s) 
\,\mathrm{d} s,
\end{align*}
here  for the sake of notation, we write $\wh{f}(\vect{k}) : = \wh{f}(k_1)\wh{f}(k_2)\wh{f}(k_3)$ and let $y(\vect{n}) = (y(n_1),y(n_2),y(n_3))$. The remaining goal (for the triple correlation) is to show $\cE(N)$ converges to $\vect{E}\:(f^3)$ as $N \to \infty$.

\subsection{Combinatorial Preparations}\label{sec: combinatorial prep}

This process of completing the $m$-point correlation and then extracting terms to isolate a target is more complicated when $m > 3$, and involves a complicated combinatorial argument. To ease the argument we first fix some notation. 

Given the set $[m]$, let $\cP$ be a partition of $\{1, \dots, m\}$. Let $\vect{n} \in \Z^m$, then we say $\vect{n}$ is $\cP$\emph{-distinct}, if $n_i =n_j$ whenever $i$ and $j$ belong to the same partition element, and otherwise, $n_i \neq n_j$. For example if $m=6$ and $\cP = \{\{1,3\}, \{4\}, \{2,5,6\}\}$, then $\vect{n}$ is $\cP$-distinct if and only if it is of the form $\vect{n} = (a,b,a,c,b,b)$ for some distinct integers $a\neq b \neq c$. Given a partition $\cP$ of $\{1, \dots, m\}$, and a vector $\vect{n} \in \Z^m$, let
\begin{align}\label{chiP def}
  \chi_{\cP}(\vect{n}) := \begin{cases}
    1 & \mbox{ if } \vect{n}\mbox{ is $\cP$-distinct}\\
    0 & \mbox{ otherwise}.
  \end{cases}
\end{align}
Moreover, given a partition $\cP$ of $[ m]$, we say that $j\in [m]$ is \emph{isolated} if $j$ belongs to a partition element of size $1$. A partition is called \emph{non-isolating} if no element is isolated (and otherwise we say it is \emph{isolating}). For our example $\cP = \{\{1,3\}, \{4\}, \{2,5,6\}\}$ we have that $4$ is isolated, and thus $\cP$ is isolating.

\subsection{Setup of the Problem: $m$-point Correlation}

For the $m$-point correlation, we proceed in the same way as we did for the triple correlation. First, assume that for $k \le m-1$ the $k$-point correlation is Poissonian. Now consider the $m^{th}$ moment of $S_N$
\begin{align}
  \cM^{(m)}(N) &:= \int_0^1 S_N(s)^m \mathrm{d}s\notag\\
              &= \int_0^1 \sum_{\vect{n} \in [N]^m}
\sum_{\vect{k} \in \Z^m} 
\left( f(N(y(n_1)+k_1 + s)) \cdots f(N(y(n_m) +k_m + s))\right) 
\mathrm{d} s\notag\\
              &= \int_{\R} \left(\sum_{\vect{n} \in  [N]^m} 
\sum_{\vect{k}\in \Z^{m-1}} 
\left(f(N(y(n_1) +k_1 + s))\cdots f(N(y(n_{m-1}) +k_{m-1} + s)) 
f(N(y(n_m) + s))\right) \right) \mathrm{d} s\notag.
\end{align}
Next, we move the $\vect{n},k_1,\ldots,k_{m-1}$ summation outside of the integral and
thereafter change variables via
$s \mapsto  N^{-1}(s - (k_m + y(n_m)))$. As a result, we see that $\cM^{(m)}(N)$ equals
\begin{align}
              & \frac{1}{N}\sum_{\vect{n} \in  [N]^m} 
\sum_{\vect{k}\in \Z^{m-1}}
\left(\int_\R \left(f(N( y(n_1)- y(n_m)+ k_1) + s) 
\cdots f(N(y(n_{m-1})-y(n_{m}) +k_{m-1}) + s) 
f(s)\right)\mathrm{d} s \right) \notag\\
              &=  \frac{1}{N}\sum_{\vect{n} \in  [N]^m} 
\sum_{\vect{k}\in \Z^{m-1}}
F\left(N(y(n_1)-y(n_2) +k_1),
N( y(n_2)-y(n_3) +k_2), \dots 
N(y(n_{m-1})-y(n_m) +k_{m-1})\right) \notag,
\end{align} 
where 
\begin{align*}
F(z_1,z_2, \dots , z_{m-1}) := \int_{\R} f(s)f(z_1+z_2+\dots + z_{m-1}+s)f(z_2+\dots + z_{m-1}+s)\dots f(z_{m-1}+s)\,\rd s. 
\end{align*}
Note that if $f\in C_c^\infty(\R^m)$ then $F\in C_c^\infty(\R^{m-1})$. The last line is simply the \emph{completed} $m$-point correlation of $F$. Hence our goal is to show that, if we replace the sum over $\vect{n}$ by the sum over $\vect{n}$ with distinct entries, then this converges to $\expect{F} = \expect{f}^m$.

First, let us understand what we have added back in by completing the sum over $\vect{n}$, this will then allow us to write down a 'target' which will provide the desired convergence (for the triple correlation this target was $\mathbf{E}(f^3)$). Consider 
\begin{align}\label{M n def}
  \cM^{(m)}(N) = \int_0^1 \sum_{\vect{n} \in [N]^m}\sum_{\vect{k} \in \Z^m} 
\left( f(N (y(n_1) +k_1 + s)) 
\cdots f(N(y(n_m) +k_m + s))\right) \mathrm{d} s
\end{align}
and apply Poisson summation to each of the sums in $k_i$, giving
\begin{align}\label{M k def}
  \cM^{(m)}(N) = \frac{1}{N^m} \int_0^1 \sum_{\vect{n} \in [N]^m}
\sum_{\vect{k} \in \Z^m}
 \wh{f}\Big(\frac{\vect{k}}{N}\Big)
e( \vect{k}\cdot y(\vect{n})     + \vect{k}\cdot \vect{1}s) \mathrm{d} s,
\end{align}
where $y(\vect{n}) :=(y(n_1), \dots, y(n_m)$. The key insight which motivates the proceeding argument is that if in \eqref{M n def} we have that $n_i$ is distinct from all other $n_j$, then the term corresponding to this case in \eqref{M k def} will come from $k_i =0$.

To access this correspondence, in \eqref{M n def}, let us further decompose the sum over $\vect{n}$  (recall the definition of $\chi_{\cP}$ \eqref{chiP def})
\begin{align*}
  \cM^{(m)}(N) = \sum_{\cP}\int_0^1 
\sum_{\vect{n} \in [N]^m}  \chi_{\cP}(\vect{n}) 
\sum_{\vect{k} \in \Z^m} \left( f(N(y(n_1) +k_1 + s)) 
\cdots f(N(y(n_m) +k_m + s))\right) \mathrm{d} s
\end{align*}
where, the sum over $\cP$ is over distinct partitions of $\{1, \dots , m\}$. Clearly, the $m$-point correlation corresponds to the trivial partition $\cP_0:=\{\{1\}, \{2\}, \dots, \{m\}\}$. All of the other terms come from completing the sum. Given a partition $\cP$, let
\begin{align*}
  \cM_{\cP}(N) := \int_0^1 \sum_{\vect{n} \in [N]^m}  
\chi_{\cP}(\vect{n}) \sum_{\vect{k} \in \Z^m} 
\left( f(N(y(n_1) +k_1 + s)) \cdots 
f(N(y(n_m) +k_m + s))\right) \mathrm{d} s.
\end{align*}

Now consider the sum \eqref{M k def}, and perform a decomposition on the $\vect{k}$ variable:
\begin{align*}
  \cM^{(m)}(N) =  \expect{f}^m + \sum_{j=2}^{m} {m \choose m-j}
\wh{f}(0)^{m-j}\frac{1}{N^j} \int_0^1 \sum_{\vect{n} \in [N]^j}
\sum_{\substack{\vect{k}\in (\Z^{\ast})^j}} 
\wh{f}\Big(\frac{\vect{k}}{N}\Big)
e( y(\vect{n})\cdot \vect{k}    + \vect{k}\cdot \vect{1}s) \mathrm{d} s,
\end{align*}
that is, we fix a $j$ and choose $m-j$ of the $k_i$ components to be equal to $0$. Note that $j$ cannot be equal to $1$ since the integral in $s$ forces $k_1 + \dots + k_m = 0$, therefore we cannot have only one $k_i \neq 0$. Let
\begin{align*}
  \cK_{j}(N):={m \choose m-j}\wh{f}(0)^{m-j}\frac{1}{N^j} 
\int_0^1 \sum_{\vect{n} \in [N]^j}
\sum_{\substack{\vect{k} \in (\Z^{\ast})^j}}
 \wh{f}\Big(\frac{\vect{k}}{N}\Big)
e(\vect{k}\cdot y(\vect{n})   + 
\vect{k}\cdot \vect{1}s) \mathrm{d} s
\end{align*}
Note that  $\cK_{0}(N):= \wh{f}\left(0\right)^m = \expect{f}^m$.

The following proposition is enough to prove Theorem \ref{thm:main}
\begin{proposition}\label{lem:MP = KP}
  Fix $j \in \{0,2,3, \dots, m\}$ we have that
  \begin{align}\label{MP = KP}
    \lim_{N \to \infty} \cK_{j}(N) = \lim_{N \to \infty} \sum_{\substack{\cP\\m-j \mbox{ iso.}}} \cM_{\cP}(N),
  \end{align}
  where the sum ranges over the partitions $\cP$ with $m-j$ many isolated points. 
\end{proposition}
This is enough to prove Theorem \ref{thm:main} since we have that the $m$-point correlation is given by 
\begin{align*}
  \cM_{\cP_0}(N) &= \cM^{(m)}(N) - \sum_{\cP \neq \cP_0} \cM_{\cP}(N)\\
           &= \sum_{j\in \{0, 2, 3, \dots, m\}} \cK_{j}(N) - \sum_{j=2}^{m} \sum_{\substack{\cP\\ m-j \mbox{ iso.}}} \cM_{\cP}(N)\\
  &= \cK_{0}(N) = \expect{f}^{m}
\end{align*}
(note that it is impossible to have all but $1$ coordinate be isolated, since a non-isolated coordinate must be in a partition element with another non-isolated coordinate).

In fact, it is enough to restrict to non-isolating partitions. Let $\mathscr{P}_m$ denote the set of non-isolating partitions of $[m]$.
\begin{lemma}\label{lem:MP = KP non-isolating}
  We have that
  \begin{align}\label{MP = KP non-isolating}
    \lim_{N \to \infty} \cK_{m}(N) = \lim_{N \to \infty} \sum_{\cP\in \mathscr{P}_m}\cM_{\cP}(N).
  \end{align}
\end{lemma}
The proof of Lemma \ref{lem:MP = KP non-isolating} is the content of Section \ref{s:Triple Correlation}. Let us assume it is true for the time being, and show that Proposition \ref{lem:MP = KP} follows.

\vspace{4mm} 
\begin{proof}[Proof that Lemma \ref{lem:MP = KP non-isolating} 
implies Proposition \ref{lem:MP = KP}]
  The proof for $m=3$ is clear from Subsection \ref{ss:Setup 3}. 
Take $m > 3$ and assume Lemma \ref{lem:MP = KP non-isolating} 
holds for all values of the correlation level less than, or equal to $m$. 
Assume $j <m$ and consider
  \begin{align*}
    \cK_j(N) = {m \choose m-j}\wh{f}(0)^{m-j}\frac{1}{N^j}
 \int_0^1 \sum_{\vect{n} \in [N]^j}
\sum_{\substack{\vect{k} \in (\Z^{\ast})^j}} 
\wh{f}\Big(\frac{\vect{k}}{N}\Big)
e( \vect{k}\cdot y(\vect{n})    + 
\vect{k}\cdot \vect{1}s) \mathrm{d} s,
  \end{align*}
  we can use Lemma \ref{lem:MP = KP} for $m =j$ to deduce that
  \begin{align}\label{K to non-isoM}
    \lim_{N \to \infty}\cK_j(N) = {m \choose m-j} \expect{f}^{m-j}\sum_{\cP\in \mathscr{P}_j} \lim_{N \to \infty}\cM_{\cP}(N).
  \end{align}
\end{proof} 
It remains to prove \eqref{MP = KP non-isolating}, or equivalently
\begin{gather}
  \lim_{N \to \infty}\cE(N) =  \sum_{\cP\in \mathscr{P}_m}  \expect{f^{\abs{P_1}}}\cdots \expect{f^{\abs{P_d}}}.\label{m target}
\end{gather}
where we have labeled the partition $\cP = (P_1, P_2, \dots, P_d)$, and $\abs{P_i}$ is the size of $P_i$, and where
\begin{align*}
  \cE(N):= \frac{1}{N^m} \int_0^1 \sum_{\vect{n} \in [N]^m}\sum_{\substack{\vect{k} \in (\Z^{\ast})^m}} \wh{f}\left(\frac{\vect{k}}{N}\right)e(\alpha \vect{k}\cdot \vect{n}^\theta  + \vect{k}\cdot \vect{1}s) \mathrm{d} s.
\end{align*}
The remainder of the paper is devoted to proving \eqref{m target}.

\subsection{Dyadic Decomposition} 
It is convenient to decompose the sums over $n$ and $k$ into dyadic ranges in a smooth manner. Given $N$, we let $Q>1$ be the unique integer with $e^{Q}\leq N < e^{Q+1}$. Now, we describe a smooth partition of unity which approximated the indicator function of $[1,N]$. Strictly speaking, these partitions depend on $Q$, however we suppress it from the notation. Furthermore, since we want asymptotics of $\mathcal{M}^{(m)}(N)$, we need to take a bit of care at the right end point of $[1,N]$, a tighter than dyadic decomposition is needed. Let us make this precise. For $0\le q < Q$ we let $\mathfrak{N}_q$ denote a smooth function for which
\begin{align*}
  \operatorname{supp}(\mathfrak{N}_q) \subset [e^{q}/2, 2 e^q)
\end{align*}
and such that $\mathfrak{N}_{q}(x) + \mathfrak{N}_{q+1}(x) = 1 $ for $x$ between $2e^{q-1}$ and $e^{q+2}/2$. Now for $q \ge Q$ we let $\mathfrak{N}_q$ form a smooth partition of unity for which 
\begin{align*}
  &\sum_{q=0}^{2Q-1} \mathfrak{N}_q (x) =\begin{cases}
  1 & \mbox{if } 1< x < e^{Q}\\
  0 & \mbox{if } x< 1/2 \mbox{ or } x > N + \frac{3N}{\log(N)}
  \end{cases},\,\mathrm{and}\\
 & \operatorname{supp}(\mathfrak{N}_q) \subset \left[\frac{e^{Q}}{2} + (q-Q)\frac{e^{Q}}{2Q}, \frac{e^Q}{Q} + (3+q-Q)\frac{e^{Q}}{2Q}\right).
\end{align*}
Let $\Vert \cdot \Vert_{\infty}$
denote the maximum norm on $\R$.
We impose the following condition on the derivatives:
\begin{align}\label{N deriv}
  \Vert \mathfrak{N}_{q}^{(t)}\Vert_{\infty} \ll \begin{cases}
    e^{-qt} & \mbox{ for } q < Q\\ 
    (e^{Q}/Q)^{-t} & \mbox{ for } Q< q,
  \end{cases}
\end{align}
for $t \le 4$. Thus 
\begin{equation}\label{eq: upper bound moment}
  \cE(N) \le \int_0^1  
  \bigg( \frac{1}{N} 
  \sum_{q=0}^{2Q-1} \sum_{n \in \Z} 
  \mathfrak{N}_{q}(n)   
  \sum_{k\neq 0} 
  \wh{f}\left(\frac{ k}{N}\right)   
  e(   k\alpha n^\theta +  ks)
  \bigg)^m \mathrm{d} s.
\end{equation}
 A similar lower bound can also be achieved by omitting some terms from the partition. 

We similarly decompose the $k$ sums, although thanks to the decay of Fourier transforms, we do not need to worry about the large $k$ values. Let $\mathfrak{K}_u$ be a smooth function such that
\begin{align*}
  \sum_{u=-U}^{U} \mathfrak{K}_u(k) =\begin{cases}
  1 & \mbox{ if } \vert k\vert  \in [1, N^{1+\varepsilon})\\
  0 & \mbox{ if } \vert k\vert < 1/2 \mbox{ or } \vert x\vert > N + 3\frac{N}{\log(N)},
  \end{cases}
\end{align*}
and the symmetry $\mathfrak{K}_{-u}(k) = \mathfrak{K}_{u}(-k)$ holds true for 
all $u,k> 0$. Additionally, we require
\begin{align*}
  &\supp(\mathfrak{K}_u) =[e^u/2, 2e^u)
  \qquad\qquad \mbox{ if } u \ge 0\,\,,\mbox{ and } \\
  &
  \Vert \mathfrak{K}_{u}^{(t)} \Vert_{\infty} \ll e^{-\abs{u}t},\qquad\qquad
  \mbox{for all } 1\le t \le 4.
\end{align*}
Therefore a central role is played by the smoothed exponential sums
\begin{equation}\label{def: E_qu}
  \cE_{q,u}(s):=
  \frac{1}{N}\sum_{k\in \Z} \mathfrak{K}_{u}(k) 
  \wh{f}\Big(\frac{k}{N}\Big)e( ks)
  \sum_{n\in \Z} \mathfrak{N}_{q}(n)e(   k\alpha n^\theta ).
\end{equation}
Notice that \eqref{eq: upper bound moment} and the rapid decay of $\wh{f}$ imply
\begin{align*}
\cE(N) \ll \bigg\Vert \sum_{u=-U}^{U}\sum_{q=0}^{2Q-1} \cE_{q,u} 
\bigg\Vert_{L^m}^m
+ o(1).
\end{align*}
Now write
\begin{align*}
  \cF(N) :=   \frac{1}{N^m} 
  \sum_{\vect{q}=0}^{2Q-1}
  \sum_{\vect{u} = -U}^{U} 
  \sum_{\vect{k},\vect{n} \in \Z^m} 
  \mathfrak{K}_{\vect{u}}(\vect{k})
  \mathfrak{N}_{\vect{q}}(\vect{n})   
  \int_0^1\wh{f}\Big(\frac{\vect{k}}{N}\Big)   
  e( \alpha\vect{k}\cdot \vect{n}^\theta  + \vect{k}\cdot \vect{1} s)\, \mathrm{d} s,
\end{align*}
where $\mathfrak{N}(\vect{n}) := \mathfrak{N}(n_1)\mathfrak{N}(n_2)\cdots\mathfrak{N}(n_m)$. Our goal will be to establish that $\cF(N) = \mathbf{E}(f^m) + o(1)$. 
Then, since we can establish the same asymptotic for the lower bound, 
we may conclude the asymptotic for $\cE(N)$. 
Since the details are identical, we will only focus on $\cF(N)$.

Fixing, $\vect{q}$, and $\vect{u}$, we let 
\begin{align*}
  \cF_{\vect{q},\vect{u}}(N) 
  & =
   \frac{1}{N^m}\int_0^1  
   \sum_{\substack{\vect{n},\vect{k} \in \Z^m}} 
   \mathfrak{N}_{\vect{q}}(\vect{n})
   \mathfrak{K}_{\vect{u}}(\vect{k})   
   \wh{f}\Big(\frac{\vect{k}}{N}\Big)   
   e(\alpha \vect{k}\cdot \vect{n}^{\theta} +  \vect{k}\cdot \vect{1}s) \mathrm{d} s.
\end{align*}

\begin{remark}
  In the proceeding sections, we will fix $\vect{q}$ and $\vect{u}$. 
Because of the way we have defined $\mathfrak{N}_q$, this implies two cases: 
$q<Q$ and $Q<q$. 
The only real difference in these two cases are the bounds in \eqref{N deriv}, 
which differ by a factor of $Q = \log(N)$. To keep the notation simple, 
we will assume we have $q <Q$ and work with the first bound. 
In practice the logarithmic correction does not affect any of the results or proofs
\end{remark}

\section{Applying the $B$-process}
\label{s:Triple Correlation}

\vspace{2mm}

Fix a small $\delta>0$. We say $(u,q)\in [N^{1+\varepsilon}]\times[2 Q]$ is \emph{degenerate} if either one of the following holds
\begin{align*}
  \alpha \theta e^{\abs{u}+(\theta-1)q}<1/10,
  \,\, \mathrm{or}\,\,
  q \leq \delta Q.
\end{align*}
Otherwise $(u,q)$ is called \emph{non-degenerate}. Let $\mathscr{G}(N)$ denote the set of all non-degenerate pairs $(u,q)$. In this section it is enough to suppose that $u>0$ (and therefore $k>0$). Next, we show that degenerate $(u,q)$ are negligible.  If $\alpha \theta e^{u+(\theta-1)q}<1/10$,  then the Kusmin--Landau estimate (see \cite[Corollary 8.11]{IwaniecKowalski2004}) implies
\begin{align*}
  \sum_{n\in \Z} 
  \mathfrak{N}_{q}(n)      
  e( k\alpha n^\theta) \ll \frac{1}{k e^{(\theta-1)q}},
\end{align*}
and hence
\begin{align*}
  \Vert \cE_{q,u} \Vert_\infty 
  \ll \frac{1}{N} 
     \sum_{k\asymp e^{u}}
     \frac{e^{(1-\theta)q}}{k}
  \ll \frac{e^{(1-\theta)q}}{N}u\ll N^{-\theta+\varepsilon}.
\end{align*}
Now suppose $q \leq \varepsilon Q$. Expanding the $m^{th}$-power, evaluating the $s$-integral and trivial estimation yield
\begin{align*}
  \Vert \cE_{q,u} \Vert_{L^m}^m
  \ll \frac{1}{N^m} \#\{k_1,\dots,k_m\asymp e^u: k_1 +\dots +k_m = 0 \}N^{m\delta}
  \ll N^{m\delta-1+\varepsilon}.
\end{align*}
The upshot is that there exists a constant $\rho=\rho(\theta) >0$ so that
\begin{align*}
  \bigg \Vert  \sum_{(u,q)\in [N^{1+\varepsilon}]\times[2Q] \setminus \mathscr{G}(N)}\cE_{q,u} 
  \bigg\Vert_{L^m}^m \ll N^{-\rho},
\end{align*}
and the triangle inequality implies
\begin{equation}\label{eq: cal F reduced}
  \mathcal{F}(N) = 
  \bigg\Vert \sum_{(u,q) \in \mathscr{G}(N)}\cE_{q,u} \bigg\Vert_{L^m}^m + O(N^{-\rho}).
\end{equation}

\subsection{First application of the $B$-Process}
\label{ss:B proc trip}
Now we are ready to apply the $B$-process to shorten the $n$-summation in $\cE_{q,u}(s)$. 
To that end, assume $k>0$ and let
\begin{align*}
  \phi(k,r): = \beta k^{\Theta}r^{1-\Theta},
\end{align*}
where
\begin{align*}
  \beta := \alpha^\Theta(\theta^{\Theta-1}-\theta^{\Theta}),
\end{align*}
note that $\beta<0$, and thus we will flip the sign of the phase function by applying the $B$-process. To simplify the analysis of signs in the two different cases $u>0$ and $u<0$, we make the following observation.
We can suppose that $f$ is an even function;
thus $\widehat{f}$ is even and real-valued.
Hence, 
\begin{equation}\label{eq: conjugation}
  \mathcal{E}_{q,-u}(s) = \overline{\mathcal{E}_{q,u}(s)}
\end{equation} 
holds for all $s\in \R$ which reduces the discussion of the case $u<0$  to the case $u>0$. The next lemma states that $\cE_{q,u}$ is approximated suitably well by 
\begin{align*}
  \cE^{(B)}_{q,u}(s) := 
  \frac{c_1 e(-1/8)}{N}\sum_{k\geq 0} \mathfrak{K}_{ u}(k)\wh{f}\Big(\frac{ k}{N}\Big) e(ks)
  \sum_{r\geq 0} \mathfrak{N}_{q}((\alpha\theta k/r)^{\Theta})\frac{k^{\frac{\Theta}{2}}}{r^{\frac{\Theta+1}{2}}}e(\phi(k,r)),
\end{align*}
where 
\begin{align*}
c_1 :=  \sqrt{\Theta(\alpha\theta)^{\Theta}}.
\end{align*}

\begin{lemma}\label{lem:E B triple}
   If $u>0$, then 
   $   \Vert \cE_{q,u} -  \cE^{(B)}_{q,u}\Vert_{\infty} = O(N^{-\varepsilon})$
   uniformly for all non-degenerate $(u,q)\in \mathscr{G}(N)$.
\end{lemma}

\begin{proof}
  Fix $k\asymp e^{\uab}$. Let $[a,b]:=\mathrm{supp}(\mathfrak{N}_{q})$, $\Phi_r(x) :=  k\alpha x^\theta - rx$, and $m(r):=\min \{\vert \Phi'_r(x)\vert: x\in [a,b]\}$. By Poisson summation and partial integration, 
  \begin{align*}
    \sum_{n\in \Z} 
    \mathfrak{N}_{q}(n)      
    e(k\alpha n^\theta )
    = \sum_{r\in \Z} \int_{-\infty}^{\infty}\mathfrak{N}_{q}(x)e( \Phi_r (x))\rd x
    = M(k) + O(N^{-100}+\mathrm{Err}(k)),
  \end{align*}
  where $M(k)$ (resp. $\mathrm{Err}(k)$) gathers the contribution of all $r\in \Z$ with $m(r)=0$ (resp. of $0< m(r)<N^{\varepsilon}$). Next, we evaluate $M(k)$. Taking $\Psi (x):= \mathfrak{N}_{q}(x), \Lambda_{\Phi_r}:= e^{\uab+q\theta}$, and $\Omega_{\Phi_r}=\Omega_{\Psi} :=  e^{q}$, Lemma \ref{lem: stationary phase} applies. The unique critical point $x_r$ of $\Phi_r$ is given by $x_r:=(\alpha\theta k/r)^{\Theta}$. Using $1+\theta \Theta = \Theta$ shows that $\Phi_r (x_r) =\phi(k,r)$ and
  \begin{align*}
    \abs{\Phi_r'' (x_r)} = \alpha k 
    \theta (\theta-1)
    \bigg(\frac{\alpha\theta k}{r}\bigg)^{(\theta-2)\Theta} =c_1^{-2}
    \frac{r^{\Theta+1}}{k^{\Theta}}.
  \end{align*}
  To ease notation, $\Phi(x):=\Phi_r(x)$.  Since $(u,q)$ is non-degenerate, $\Lambda_{\Phi}/\Omega_\Phi = e^{\uab+(\theta-1)q} >1/(10\alpha \theta)$. Thus
  \begin{equation}\label{eq: stationary phase main term}
    M(k) = c_1e(-1/8) \sum_{r\in \Z} 
    \mathfrak{N}_q\left(k,r\right) e(\phi(k,r)) + 
    O\big( \Lambda_{\Phi}^{-1/2+O(\varepsilon)} \big).
  \end{equation}
  To bound $\mathrm{Err}(k)$, notice $m(r) = \min(\vert \Phi_r'(a)\vert, \vert \Phi_r'(b)\vert)$. Hence there are $O(N^\varepsilon)$ many $r$ with $0< m(r)<N^{\varepsilon}$. By swapping $a$ and $b$, if needed, we have $m(r)=\vert \Phi_r'(a)\vert \geq \Vert \alpha \theta a^{\theta-1} k\Vert$. Lemma \ref{lem: van der Corput's lemma} (for $i=1,2$) yields
  \begin{align*}
    &\int_{-\infty}^{\infty} \Psi(x) e( \Phi_r (x))\rd x
    \ll 
    \min\bigg(\frac{1}{m(r)},\frac{1}{\sqrt{e^{u+q(\theta-2)}}}\bigg), \qquad \mbox{ thus }\\
    &\mathrm{Err}(k) \ll N^{\varepsilon}
    \min\bigg(\frac{1}{\Vert \alpha \theta a^{\theta-1} k\Vert},
    \frac{1}{\sqrt{e^{\uab+q(\theta-2)}}}\bigg).
  \end{align*}
  Next we observe that whenever $\omega,\Omega>0$ satisfy $0<10\omega<\Omega<1/10$, then 
  \begin{align*}
    \sum_{k\asymp e^{\uab}}
    \min\bigg(\frac{1}{\Vert \omega k\Vert},
    \frac{1}{\Omega}\bigg) \ll \uab e^{\uab}.
  \end{align*}
  Here, we take $\omega:= \alpha \theta a^{\theta-1}$ and $\Omega:= e^{\frac{u+q(\theta-2)}{2}}$. Combining the previous two bounds implies 
  \begin{equation}\label{eq: low frequency error term}
    \sum_{k\asymp e^{\uab}}\mathrm{Err}(k)=O(N^{1-10\varepsilon}), \,\,
    \mathrm{provided}\,\, \uab < (1-10\varepsilon)\log N.
  \end{equation}
  
  On the other hand, suppose $\uab \geq (1-10\varepsilon)\log N$. The mean value theorem gives us the lower bound $\Phi^\prime(a+a^{1-\theta +16\varepsilon}) - \Phi^\prime(a) \gg ka^{-1+16\varepsilon} \gg N^{4\varepsilon}$. Thus, by monotonicity, $\Phi^\prime(x) \gg N^{4\varepsilon}$ for $x \in [a+a^{1-\theta+16\varepsilon}, b]$. Due to \eqref{N deriv}, we infer $\Psi(a+a^{1-\theta+16\varepsilon}) \ll e^{- q(\theta-16\varepsilon)}$. Hence
  \begin{align*}
    \int_{-\infty}^{\infty}\Psi(x)e( \Phi_r (x))\rd x
    & \ll
    \int_{a}^{a+a^{1-\theta+16\varepsilon}} \Psi(x)e( \Phi_r (x)) \rd x
    +N^{-3\varepsilon}\\
    & \ll 
    N^{-2\varepsilon}
    \bigg(\min\bigg(\frac{1}{\Vert \alpha \theta a^{\theta-1} k\Vert},
    \frac{1}{\sqrt{e^{\uab+q(\theta-2)}}}\bigg)+1 \bigg).
  \end{align*}
  Arguing as before, we conclude
  \begin{equation}\label{eq: high frequency error term}
    \sum_{k\asymp e^{\uab}} \mathrm{Err}(k)=O(N^{1-\varepsilon}), \,\,
    \mathrm{provided}\,\, (1-10\varepsilon)\log N \leq \uab \ll \log N.
  \end{equation}
  The proof is completed by summing \eqref{eq: stationary phase main term}, \eqref{eq: low frequency error term}, and \eqref{eq: high frequency error term} against $ N^{-1} \mathfrak{K}_{u}(k)  \wh{f}( k/N) e(ks)$ for $k\geq 0$.
\end{proof}

  \subsection{Second Application of the $B$-Process}

  Next we apply the $B$-process to shorten $k$-summation within $\cE^{(B)}_{q,u}$. To this end, define (for $u>0$) 
\begin{align*}
  \cE^{\mathrm{(BB)}}_{q,u}(s) := 
  \frac{c_1}{N}\sum_{r \geq 0}
  \sum_{h \geq 0} 
  \widehat{f}\left(\frac{\mu}{N}\right) 
  \mathfrak{N}_{q}(r,h)
  \mathfrak{K}_{u} (\mu )
  \frac{\mu^{\Theta/2}}{\sqrt{\phi_{\mu\mu}(\mu,r)}}
  e(c (h- s)^{1/\theta}r)
\end{align*}
where 
\begin{gather}\label{mu def}
  \begin{gathered}
    \phi_{\mu\mu}(\mu,r):= \frac{\partial^2}{\partial h^2}\phi(h,r)\Big\vert_{h=\mu},
    \qquad \qquad
    \mathfrak{N}_{q}(h,s) := \mathfrak{N}_{q}(      (\alpha\theta c_0  (h-  s)^{1/(\Theta-1)})^{\Theta}),\\
    \mu:=\mu(h,r,s) := c_0 r (h-  s)^{1/(\Theta-1)}
  \end{gathered}
\end{gather}
and where the two constants $c,c_0$, depend only on $\alpha$, and $\theta$ but do not play a role in what follows.
Then we have the following lemma

\begin{lemma} \label{B proc twice}
  If $u>0$, then
   $   \Vert \cE^{\mathrm{(BB)}}_{q,u} - \cE^{\mathrm{(B)}}_{q,u} \Vert_{\infty} 
=O(N^{-\varepsilon})$
  uniformly for any non-degenerate $(u,q)\in \mathscr{G}(N)$.
\end{lemma}
\begin{proof}
Fix $r\asymp e^{u+q(\theta-1)}$. For ease of exposition, let
\begin{gather*}
  g(k):= \wh{f}\big(k/N\big) \left(\frac{k}{e^{u}}\right)^{\Theta/2},
  \qquad \qquad
  \Psi(x):= \mathfrak{K}_{u}(x) \mathfrak{N}_{q}((\alpha\theta k/r)^{\Theta}) g(x),\qquad\qquad
  \Phi_h(x) :=  \phi(x,r) -  x(h -  s),
\end{gather*}
and $m(h):=\min \{\vert \Phi'_h(x)\vert: x\in [a,b]\}$. By Poisson summation
  \begin{align*}
    \sum_{k\geq 0} \mathfrak{K}_{ u}(k) \wh{f}\Big(\frac{ k}{N}\Big) e( ks)\mathfrak{N}_{q}((\alpha\theta k/r)^{\Theta}) k^{\frac{\Theta}{2}} e(\phi(k,r))
    =e^{u\Theta/2}\sum_{h\in \Z }\int_{\R} \Psi(x) e\left( \Phi_h(x)\right) \,\mathrm{d}x.
\end{align*}
By partial integration the right hand side equals
\begin{align*}
M(r) + O(N^{-100}+\mathrm{Err}(r))
\end{align*}
  where $M(r)$ (resp. $\mathrm{Err}(r)$) gathers the contribution of all $h\in \Z$ with $m(h)=0$ (resp. of $0< m(h)<e^{{\abs{u}}\varepsilon}$).

  We evaluate $M(r)$ by Lemma \ref{lem: stationary phase} (by scaling the amplitude by a constant factor) with the specifications
  \begin{align*}
    \Lambda_\Phi:= e^{\uab +q\theta}, \qquad
    \Omega_\Phi = \Omega_{\Psi} :=  e^{\uab}.
  \end{align*}
  Note that $\mu$ is the unique critical point of
  $\Phi_h$. An application of Lemma \ref{lem: stationary phase} implies (note that $\beta <0$, thus the phase is negative)
  \begin{align*}
    M(r) =   e(1/8)
  \sum_{h\in \Z} 
  \widehat{f}\left(\frac{\mu}{N}\right) 
  \mathfrak{N}_{q}(r,s)
  \mathfrak{K}_{u} (\mu )
  \frac{\mu^{\Theta/2}}{\sqrt{\phi_{\mu\mu}(\mu,r)}}
  e(c (h-s)^{1/\theta}r) + O(e^{-\uab/2-3q\theta/2}).
  \end{align*}
  To estimate $\mathrm{Err}(r)$, we proceed as in the proof of Lemma \ref{lem:E B triple}. First, we observe that if $h$ is so that  $0< m(h)<e^{\uab \varepsilon}$ then the critical point $\mu$ is near one of the boundary points $a,b$. By possibly interchanging their roles, we can assume $\mu$  is near $a$, i.e. $m(a)=\vert \Phi_h(a)\vert$. Note that $\vert \Phi'_h(x)\vert\gg e^{5{\uab}\varepsilon}$ on the interval $[a+a^{1-5\varepsilon},b]$ and that $\Psi(a+a^{1-2\varepsilon}) \ll e^{-2\varepsilon}$. Hence, by Lemma \ref{lem: van der Corput's lemma} shows
  \begin{align*}
    \mathrm{Err}(r) \ll  \frac{N^{-\varepsilon}}{\sqrt{e^{-\uab+q\theta}}}.
  \end{align*}
  Thus
  \begin{align*}
    e^{u\Theta/2}\frac{1}{N}\sum_{r\asymp e^{u+q(\theta-1)}} 
    r^{-\frac{\Theta+1}{2}}
    \Err(r)\ll \frac{1}{N} e^{\uab+q(\theta-1)}  \frac{1}{\sqrt{e^{\uab+q(\theta-2)}}}
    \frac{N^{-\varepsilon}}{\sqrt{e^{-\uab +q\theta}}}
    = \frac{e^{\uab}}{N^{1+\varepsilon}} \ll N^{-\varepsilon}.
  \end{align*}
  Summing
   $ M(r)c_1 e(-1/8) N^{-1} r^{-\frac{\Theta+1}{2}}$, 
  over $r\asymp e^{u+q(\theta-1)}$ finishes the proof.
\end{proof}

We summarise how the previous lemmas transform \eqref{eq: cal F reduced}, for which let $\sigma_i := \sigma(u_i) := \frac{u_i}{\abs{u_i}}$ and $\vect{\sigma}:=(\sigma_1,\sigma_2, \dots ,\sigma_m)$. Combining \eqref{eq: conjugation} and Lemma \ref{lem:E B triple} yields  
\begin{align*}
\mathcal{F}(N) & =   \bigg\Vert \sum_{\substack{(u,q) \in \mathscr{G}(N)\\ u>0}}\cE_{q,u}
+  \sum_{\substack{(u,q) \in \mathscr{G}(N)\\ u>0}} \overline{\cE_{q,u}}
 \bigg\Vert_{L^m}^m \\
 & = \bigg\Vert \sum_{\substack{(u,q) \in \mathscr{G}(N)\\ u>0}}\cE^{(B)}_{q,u}
+  \sum_{\substack{(u,q) \in \mathscr{G}(N)\\ u>0}} \overline{\cE^{(B)}_{q,u}}
 \bigg\Vert_{L^m}^m  + O(N^{-\varepsilon/2}).
\end{align*}
Using Lemma \ref{B proc twice} and expanding the $m^{th}$-power gives
\begin{equation}\label{eq: cal F B-processed}
\mathcal{F}(N) = \sum_{\sigma_1,\ldots,\sigma_m\in \{\pm 1\} }
\sum_{\substack{(u_i,q_i) \in \mathscr{G}(N)\\ u_i>0}}
\int_{0}^{1} 
\prod_{\substack{i\leq m\\ \sigma_i >0}} 
\cE^{\mathrm{(BB)}}_{q_i, u_i}(s)
\prod_{\substack{i\leq m\\ \sigma_i <0}} 
\overline{\cE^{\mathrm{(BB)}}_{q_i, u_i}(s)}
 \rd s
+ O(N^{-\varepsilon/2}).
\end{equation}
To simplify this expression, for a fixed $\vect{u}$ and $\vect{q}$, and $\vect{\mu}=(\mu_1,\ldots,\mu_m)$ we define the function
$\mathfrak{K}_{\vect{u}}(\vect{\mu}):= 
   \prod_{i\leq m}\mathfrak{K}_{u_i}(\mu_i) $.
The functions $ \mathfrak{N}_{\vect{q}}\left(\vect{\mu},s\right)$
and $\wh{f}(\vect{\mu}/N)$ are defined in the same fashion.
Aside from the error term, the right hand side of 
\eqref{eq: cal F B-processed} splits into a sum over
    \begin{align*}
      \cF_{\vect{q},\vect{u}} := \frac{c_1^m}{N^m} 
      \sum_{\vect{r} \in \Z^m}(r_1r_2\cdots r_m)^{-(\Theta+1)/2}
      \int_0^1\sum_{\vect{h}\in \Z^m }
      \mathfrak{K}_{\vect{u}}(\vect{\mu}) 
      \mathfrak{N}_{\vect{q}}\left(\vect{\mu},s\right) 
      A_{\vect{h},\vect{r}}(s) 
      e\left( \varphi_{\vect{h},\vect{r}}(s)\right) \mathrm{d}s 
    \end{align*}
where the phase function is given by
  \begin{align*}
    \varphi_{\vect{h},\vect{r}}(s) := c \left(\sigma_1 (h_1 - s)^{1/\theta}r_1 + \sigma_2(h_2 - s)^{1/\theta}r_2 +\dots +  \sigma_m(h_m - s)^{1/\theta}r_m\right)
  \end{align*}
  and the amplitude function is
  \begin{align*}
    A_{\vect{h},\vect{r}}(s) : =  \wh{f}\left(\frac{ \vect{\mu}}{N}\right)\frac{(\mu_1\mu_2\cdots\mu_m)^{\Theta/2}}{\sqrt{\abs{\phi_{\mu\mu}(\mu_1,r_1)\phi_{\mu\mu}(\mu_2,r_2)\cdots\phi_{\mu\mu}(\mu_m,r_m)}}}.
  \end{align*}
  Note that the argument of $\wh{f}$ should be $(\mu_1\sigma_1, \dots, \mu_m\sigma_m)$ however to simplify matters we can assume (w.l.o.g) $f$ is even. Now to analyse these transformed sums, we distinguish between two cases. First, what we call the set of all $(\vect{r},\vect{h})$ the \emph{diagonal}, which is when the phase $\varphi_{\vect{h},\vect{r}}(s)$ vanishes identically. Let
  \begin{align*}
    \mathscr{A} : = \{ (\vect{r}, \vect{h}) \in \N\times \N : 
    \varphi_{\vect{h},\vect{r}}(s) = 0, \forall s \in [0,1] \},
  \end{align*}
  and let
  \begin{align*}
    \eta(\vect{r},\vect{h}):=\begin{cases}
    1 & \mbox{ if } (\vect{r},\vect{h}) \not\in \mathscr{A} \\
    0 & \mbox{ if } (\vect{r},\vect{h}) \in \mathscr{A}.
    \end{cases}
  \end{align*}
 The diagonal, as we show, contributes the main term, while
 the off-diagonal contribution is negligible (see the penultimate section).

  \section{Extracting the Diagonal}

  First, we establish an asymptotic for  the diagonal. To ease the notation, the below sums range over $\vect{q} \in [2Q]^m$, $\vect{u}\in [-U,U]$, and $\vect{r},\vect{h}\in \Z$, 
    \begin{align*}
      \cD_N &=  \frac{c_1^m}{N^m}\sum_{\vect{q},\vect{u},\vect{r},\vect{h} }(1-\eta(\vect{r},\vect{h}))(r_1r_2 \cdots r_m)^{-(\Theta+1)/2}   \int_0^1 \mathfrak{K}_{\vect{u}}(\vect{\mu}) \mathfrak{N}_{\vect{q}}\left(\vect{\mu},s\right) A_{\vect{h},\vect{r}}(s) e\left( \varphi_{\vect{h},\vect{r}}(s)\right) \mathrm{d}s  \\
      &= \frac{c_1^m}{N^m}\sum_{\vect{q},\vect{u},\vect{r},\vect{h} }(1-\eta(\vect{r},\vect{h}))(r_1r_2\cdots r_m)^{-(\Theta+1)/2}  \int_0^1 \mathfrak{K}_{\vect{u}}(\vect{\mu}) \mathfrak{N}_{\vect{q}}\left(\vect{\mu},s\right) A_{\vect{h},\vect{r}}(s)  \mathrm{d}s
    \end{align*}
    note that the phase function is $\varphi_{\vect{h},\vect{r}} $ 
    uniformly $0$ on the diagonal. 

    \begin{lemma}\label{lem:diag}
      We have
      \begin{align}\label{diag}
       \lim_{N\rightarrow \infty} \cD_N = 
       \sum_{\cP\in \mathscr{P}_m} 
       \mathbf{E}(f^{\abs{P_1}})\cdots \mathbf{E}(f^{\abs{P_d}}).
      \end{align}
      where the sum is over all non-isolating partitions of $[m]$, 
      which we denote $\cP = (P_1, \dots, P_d)$.
    \end{lemma}

    \begin{proof}
      First, we note that in $\cD_N$, we have the factor
      \begin{align*}
        \sum_{\vect{u}\in \Z^m} \mathfrak{K}_{\vect{u}}(\vect{\mu}) \wh{f} \left(\frac{\vect{\mu}}{N}\right)
      \end{align*}
      but recall that $\sum_{\vect{u}\in \Z^m} \mathfrak{K}_{\vect{u}}(\vect{\mu}) = 1$ if $\mu_i \ll N^{1+\varepsilon}$ for $i=1,2,\dots, m $. Thus, by the fast decay of $\widehat{f}$, we can add back in the larger $\vect{\mu}$ contributions (although, note that we have extracted the $\abs{\mu_i} <1/2$ contribution):
    \begin{align*}
      \cD_N 
      &= \frac{c_1^m}{N^m}\sum_{\vect{q},\vect{r},\vect{h} } \one(\abs{\mu_i} >0)(1-\eta(\vect{r},\vect{h}))(r_1r_2\dots r_m)^{-(\Theta+1)/2}   \int_0^1 \mathfrak{N}_{\vect{q}}\left(\vect{\mu},s \right) A_{\vect{h},\vect{r}}(s)  \mathrm{d}s.
    \end{align*}
    Since $\eta \neq 1$, we have that $(\vect{r},\vect{h}) \in \mathscr{A}$. That is $\varphi_{\vect{r},\vect{h}}(s)=0$. Looking at the definition, this happens precisely in the following situation: let $\cP$ be a non-isolating partition of $[m]$, we say a vector $(\vect{r},\vect{h})$ is $\cP$\emph{-adjusted} if for every $P \in \cP$ we have: $h_i = h_j$ for all $i, j \in \cP$, and $\sum_{i\in P} r_i = 0$. The diagonal is restricted to $\cP$-adjusted vectors. Now 
    \begin{align*}
      \chi_{\cP,1}(\vect{r})
      :=
      \begin{cases}
        1 & \mbox{ if $\sum_{i\in P} r_i =0$ for each $P \in \cP$}\\
        0 & \mbox{ otherwise,}
      \end{cases},
      \qquad\qquad
      \chi_{\cP,2}(\vect{h})
      :=
      \begin{cases}
        1 & \mbox{ if $h_i =h_j$ for $i,j \in P\in \cP$}\\
        0 & \mbox{ otherwise,}
      \end{cases}
    \end{align*}
    here $\chi_{\cP,1}(\vect{r})\chi_{\cP,2}(\vect{h})$ encodes the condition that $(\vect{r},\vect{h})$ is $\cP$ adjusted.
    
    Unpacking the definition of $A_{\vect{h},\vect{r}}(s)$ gives (note that $\vect{\mu} =\vect{\mu}(s)$)
    \begin{align*}
      \cD_N 
      &= \frac{1}{N^m}\frac{c_1^m}{(\beta\Theta(\Theta-1))^{m/2}}\sum_{\cP\in \mathscr{P}_m} \sum_{\vect{q},\vect{r}, \vect{h} }\chi_{\cP,1}(\vect{r})\chi_{\cP,2}(\vect{h})(r_1r_2\cdots r_m)^{-1}\\
      &\phantom{++++++++++++++++++++++}\left( \int_0^1 \mathfrak{N}_{\vect{q}}\left(\vect{\mu}, s\right)\wh{f}\left(\frac{ \vect{\mu}}{N}\right) \mu_1\mu_2\cdots \mu_m  \mathrm{d}s\right) +o(1).
    \end{align*}
    First note that the constant prefactor: 
    \begin{align*}
      \frac{c_1}{(\beta\Theta(\Theta-1))^{1/2}} = \frac{((\alpha\theta)^{\Theta}\Theta)^{1/2}}{(\alpha^{\Theta}(\theta^{\Theta-1}(1-\theta)\Theta(\Theta-1))^{1/2}} = 1. 
    \end{align*}
    Now inserting the definition of $\mu_i$ gives
    \begin{align*}
      \cD_N 
      &= \frac{1}{N^m}\sum_{\cP\in \mathscr{P}_m}\sum_{\vect{q},\vect{r}, \vect{h} } \chi_{\cP,1}(\vect{r})\chi_{\cP,2}(\vect{h})  \int_0^1 \mathfrak{N}_{\vect{q}}\left(\vect{\mu},s\right)\wh{f}\left(\frac{\vect{\mu}}{N}\right)   \prod_{i=1}^m\left(c_0 (h_i-s)^{1/(\Theta-1)}\right)  \mathrm{d}s +o(1).
    \end{align*}
    Now note that the $\vect{r}$ variable only appears in $\wh{f}\left(\vect{\mu}/N\right)$, that is
    \begin{align}\label{D intermediate}
      \begin{aligned}
      \cD_N 
      &= \frac{1}{N^m}\sum_{\cP\in \mathscr{P}_m}
      \sum_{P \in \cP}
      \sum_{\vect{q}, h }  
       \int_0^1 \mathfrak{N}_{\vect{q},P}\left(h\right)
      c_0^{\abs{P}} h^{\frac{\abs{P}}{\Theta-1}}
      \sum_{\substack{\vect{r}\in \Z^{\abs{P}}\\r_i \neq 0}}
      \chi(\vect{r}) \wh{f}\left(\frac{
      c_0 h^{\frac{1}{\Theta-1}} }{N} 
      \vect{r}\right)     \mathrm{d}s(1+ o(1)),
      \end{aligned}
    \end{align}
    where $\chi(\vect{r})$ is $1$ if $\sum_{i=1}^{\abs{P}} r_i = 0$ 
    and where $\mathfrak{N}_{\vect{q},P}(h) = \prod_{i\in P}\mathfrak{N}_{q_i} (\abs{ \alpha \theta c_0 (h -  s)^{1/\Theta -1 }}^\Theta ) $. Focusing on the sums in $r_1$ and $r_2$, we can apply Euler's summation formula (\cite[Theorem 3.1]{Apostol1976}) to conclude that
    \begin{align*}
      \sum_{\substack{\vect{r}\in \Z^{\abs{P}}\\r_i \neq 0}}\chi(\vect{r}) \wh{f}\left(\frac{c_0(h-s)^{1/(\Theta-1)} }{N} \vect{r}\right)   
      =  
      \int_{\R^{\abs{P}}} \chi(\vect{x})\wh{f}\left(\frac{c_0h^{1/(\Theta-1)}}{N} \vect{x}\right)\mathrm{d}\vect{x}\left(1+ o(1) \right).
    \end{align*}
    Because of the condition imposed by $\mathfrak{N}_{\vect{q},P}(\vect{h})$ 
    we have $h_i^{1/\theta} \ll N$ for every $i = 1, \dots, d$, 
    therefore $h_i^{1/(\Theta-1)} \ll N^{1-\theta}$. Changing variables 
    $\vect{x}\mapsto h (c_0^{\frac{-1}{\Theta-1}}N)^{-1}\vect{x}$
    yields
    \begin{align*}
      & \int_{\R^{\abs{P}}} \chi(\vect{x})\wh{f}\left(\frac{c_0 h^{1/(\Theta-1)}}{N} \vect{x}\right)\mathrm{d}\vect{x}
      =\frac{N^{\abs{P}-1}}{c_0 (h_i-s)^{(\abs{P}-1)/(\Theta-1)}}  \int_{\R^{\abs{P}}} \chi_{\cP}(\vect{x}) \wh{f}\left(\vect{x}\right)\mathrm{d}\vect{x} \left(1+ O\left(N^{-\theta}\right)\right)
    \end{align*}
    here we have used that, because of $\chi(\vect{x})$, we have $x_{\abs{P}} = -\sum_{i=1}^{\abs{P}-1}x_i$ and is therefore fixed. This is why the leading factor is taken to the $\abs{P}-1$ power. Plugging this into our \eqref{D intermediate} gives
    \begin{align*}
      \cD_N 
      &= \frac{1}{N^d}\sum_{\cP\in \mathscr{P}_m}\sum_{P \in \cP}
      \sum_{\vect{q}, h }    \mathfrak{N}_{\vect{q},P}\left(h\right)\left(c_0 h^{1/(\Theta-1)}\right) 
      \int_{\R^{\abs{P}-1}} \wh{f}(x_1, \dots, x_{\abs{P}-1}, -\vect{x}\cdot \vect{1}) \rd \vect{x} (1+ o(1))
    \end{align*}
    We claim that the quantity in the first line is exactly $1+o(1)$.

    By the Euler's summation formula
    \begin{align*}
      \sum_{\vect{q}, h }  \bigg( \mathfrak{N}_{\vect{q}}\left(h-s\right)c_0 (h-s)^{1/(\Theta-1)}\bigg) &= \theta \left(\frac{N (\beta \Theta)^{1/\theta}/(\alpha\theta)^{\Theta }}{(\beta\Theta)^{1/(\Theta-1)}}\right)(1+o(1))\\
      &= N \theta \left(\frac{ \beta \Theta}{(\alpha\theta)^{\Theta }}\right)(1+o(1))=  N(1+o(1))
    \end{align*}
    Thus, we arrive at
    \begin{align*}
      \cD_N 
      &= \sum_{\cP\in \mathscr{P}_m}\sum_{P \in \cP}
      \left( \int_{\R^{\abs{P}-1}} \wh{f}(x_1, \dots, x_{\abs{P}-1}, -\vect{x}\cdot \vect{1}) \rd \vect{x} \right)(1+ o(1))
    \end{align*}
    Finally consider
    \begin{align*}
      \int_{\R^{\abs{P}-1}} \wh{f}(x_1, \dots, x_{\abs{P}-1}, -\vect{x}\cdot \vect{1}) \rd \vect{x}
      =
      \int_{\R^{\abs{P}-1}} \wh{f}(x_1)\wh{f}( x_{\abs{P}-1})\wh{f}( -\vect{x}\cdot \vect{1}) \rd \vect{x}
  \end{align*}
  If we focus on the integral in $x_1$, this is simply a convolution of Fourier transforms, using that the convolution of Fourier transforms is the Fourier transform of the same functions multiplied together we conclude that
  \begin{align*}
    \int_{\R^{\abs{P}-1}} \wh{f}(x_1, \dots, x_{\abs{P}-1}, -\vect{x}\cdot \vect{1}) \rd \vect{x}
    =
    \expect{f^{\abs{P}}}
  \end{align*}
  which leads exactly to \eqref{diag}.

      \end{proof}

  \section{Bounding the Off-Diagonal}

  It remains to bound the off-diagonal contribution, for fixed $\vect{r}$ we thus want to bound
  \begin{align*}
    \cO_N : = \int_0^1\sum_{\vect{h}\in \Z^m } \eta(\vect{r},\vect{h})\mathfrak{K}_{\vect{u}}(\vect{\mu}) \mathfrak{N}_{\vect{q}}\left(\vect{\mu},s\right) A_{\vect{h},\vect{r}}(s) e\left( \varphi_{\vect{h},\vect{r}}(s)\right) \mathrm{d}s
  \end{align*}
  which requires exploiting the $s$ integral. We write the new amplitude function as 
  \begin{align*}
    &\wt{A}_{\vect{h},\vect{r}}(s) : =     \frac{(\mu_1\mu_2\cdots \mu_m)^{\Theta/2}}{\sqrt{\Phi_{\mu\mu}(\vect{\mu},\vect{r})}}\mathfrak{K}_{\vect{u}}(\vect{\mu}) \mathfrak{N}_{\vect{q}}\left(\vect{\mu},s\right)\wh{f}\left(\frac{\vect{\mu}}{N}\right) .
  \end{align*}
  Further write
  \begin{align*}
    &\cO_N \ll \sum_{\vect{h}\in \Z^m } \eta(\vect{r},\vect{h}) I(\vect{h},\vect{r})\mathrm{d}, \qquad \mbox{ where } \qquad I(\vect{h},\vect{r}):= \int_0^1 \wt{A}_{\vect{h},\vect{r}}(s) e\left( \varphi_{\vect{h},\vect{r}}(s)\right) \mathrm{d}s.
  \end{align*}
  By relabeling and redefining variables, we may write 
  \begin{align*}
    \varphi_{\vect{h},\vect{r}}(s)  = \sum_{\ell \le l} cr_\ell (h_\ell -s)^{1/\theta}  - \sum_{l < \ell \le L} c r_{\ell}(h_\ell -s)^{1/\theta}
  \end{align*}
  where $L\le m$ and $h_\ell$ are pairwise distinct. Now the following proposition establishes a bound for $I$.
  
  \begin{proposition}\label{prop:int s}
    Let $\varphi$ be as above, then 
    \begin{align}\label{exp varphi}
      I(\vect{h},\vect{r}) \ll \mathfrak{K}_{\vect{u}}(\vect{\mu}_0) \mathfrak{N}_{\vect{q}}\left(\vect{\mu}_0,s\right) \frac{e^{u_1+\dots+u_m}}{(r_1r_2\cdots r_m)^{(1-\Theta)/2}}\max_{t\le L} \left(e^{-u_t}e^{\theta((L-1)q_t+\sum_{t\neq \ell\leq L}q_\ell)}\prod_{\substack{\ell\le L\\\ell \neq t}} \abs{h_{\ell} - h_t}^{-1} \right)^{1/L}.
    \end{align}
    as $N \to \infty$. Where $\mu_{0,i} := r_i(h_i/\beta\Theta)^{1/(\Theta-1)}$, that is $\mu_i$ with $s=0$. Where the implicit constants do not depend on in $\vect{h}$ or $\vect{r}$ provided $\eta_{\vect{r}}(\vect{h})\neq 0$.
  \end{proposition}
To prove Proposition \ref{prop:int s} we aim to show that, at least one of the first $j$ derivatives $\varphi^{(j)}$ is of size $N^{1-(j-1)\theta}$. Then we can use van der Corput's lemma to gain an absolute power of $N$. Importantly, note that $\varphi$ is only zero function when $\eta(\vect{r},\vect{h}) =0$.

The first $L$-derivatives are simultaneously small if 
\begin{equation}
  a_{j}(s):=\varphi^{(j)}(s)=\sum_{\ell\leq l}c r_{\ell}{1/\theta\choose j}(h_{\ell}-s)^{1/\theta-j}-\sum_{l<\ell\leq L}c r_{\ell}{1/\theta\choose j}(h_{\ell}-s)^{1/\theta-j}\label{eq: vector w_s}
\end{equation}
is in a small interval, say, $[-N^{\delta},N^{\delta}]$. We will show that this cannot happen for $\delta>0$ sufficiently large to achieve \eqref{exp varphi}. To that end, recast (\ref{eq: vector w_s}) as the matrix-vector equation $\mathbf{a}=M\mathbf{b}$ in $\mathbb{R}^{L}$ where
\begin{equation}
  a_{\ell}:=a_{\ell}(s)\quad(\ell\leq j),\qquad m_{j,\ell}:={1/\theta\choose j}(h_{\ell}-s)^{-j}\label{def: a_ell and matrix M}
\end{equation}
and
\begin{align*}
  b_{\ell}:=\begin{cases}
  c r_{\ell}(h_{\ell}-s)^{1/\theta}, & \mathrm{if}\,\ell\leq l,\\
  -c r_{\ell}(h_{\ell}-s)^{1/\theta}, & \mathrm{if}\,l<\ell\leq L.
  \end{cases}
\end{align*}

The key idea is to show that the spectral norm 
$\Vert \cdot \Vert_{\mathrm{spec}}$ of $M^{-1}$, 
i.e. the operator norm induced by the the Euclidean norm $\Vert\cdot\Vert_{2}$, 
is not to large. Once this is done we can argue via
\begin{equation}
\mathbf{b}=M^{-1}\mathbf{a}\Longrightarrow\left\Vert \mathbf{b}\right\Vert 
\leq\left\Vert M^{-1}\right\Vert _{\mathrm{spec}}\left\Vert \mathbf{a}\right
\Vert _{2}\Longrightarrow\left\Vert \mathbf{a}\right\Vert _{2}
\geq\frac{\left\Vert \mathbf{b}\right\Vert }
{\Vert M^{-1}\Vert_{\mathrm{spec}}}\label{eq: lower bound on norm of vector b}.
\end{equation}
Because the components of vector $\mathbf{b}$ have size 
\[
\approx\max_{\ell\leq L} r_{\ell}(h_{\ell}-s)^{1/\theta}\asymp
\max_{\ell\leq L}e^{\theta q_\ell +  u_\ell}
\]
this will be enough to show that choosing $\delta\approx1-L\theta>0$
we cannot have $\mathbf{a}\in[-N^{\delta},N^{\delta}]^{L}$. 

\begin{lemma}
\label{lem: Van der Mond inverse}Let $\tau_{1},\ldots,\tau_{L}$
be distinct real numbers, and 
\[
V:=V(\tau_{1},\ldots,\tau_{L}):=\begin{pmatrix}\tau_{1} & \ldots & \tau_{L}\\
\vdots & \ddots & \vdots\\
\tau_{1}^{L} & \ldots & \tau_{L}^{L}
\end{pmatrix}.
\]
Then $V$ is invertible and $V^{-1}=:(v_{t,T})_{t,T\leq L}$ satisfies
\[
v_{t,T}=(-1)^{T-1}\biggl(\tau_{t}\prod_{\underset{l\neq t}{l\leq L}}(\tau_{l}-\tau_{t})\biggr)^{-1}\sum_{\underset{\ell_{1},\ell_{2},\ldots,\ell_{L-T}\neq l}{\ell_{1}<\ell_{2}<\ldots<\ell_{L-T}\leq L}}\tau_{\ell_{1}}\ldots\tau_{\ell_{L-T}}.
\]
\end{lemma}

\begin{proof}
Linear Algebra. Note that $V$ is essentially a scaled Vandermonde matrix.
\end{proof}
With this lemma at hand, we have

\begin{lemma}\label{lem: operator norm bound}
  If $M$ is given by (\ref{def: a_ell and matrix M}), then 
  \[
  \Vert M^{-1}\Vert_{\mathrm{spec}}
  \ll      
e^{\theta ((L-1)q_t + \sum_{\ell\leq L}q_\ell))}
\max_{\ell\neq t\leq L}\bigg(  
\prod_{\underset{\ell\neq t}{\ell\leq L}}
\abs{h_{\ell}-h_t}^{-1}\bigg).
  \]
\end{lemma}
\begin{proof}
  Let us decompose $M$ via $ M=M_{\mathrm{Van}}M_{\mathrm{diag}}$ where 
  \[
  M_{\mathrm{Van}}:=((h_{\ell}-s)^{-j})_{j,\ell\leq L},\qquad M_{\mathrm{diag}}:=\mathrm{diag}\left(\begin{pmatrix}\vartheta\\
1
\end{pmatrix},\ldots,\begin{pmatrix}\vartheta\\
L
\end{pmatrix}\right)
\]
(with $\mathrm{diag}$ denoting a diagonal matrix). Clearly, 
\[
\Vert M^{-1}\Vert_{\mathrm{spec}}\ll
\Vert M_{\mathrm{Van}}^{-1}\Vert_{\mathrm{spec}}.
\]
Taking $\tau_{\ell}:=(h_{\ell}-s)^{-1} < 1$ 
in Lemma \ref{lem: Van der Mond inverse}
and bounding the spectral norm by the maximum norm,  
\begin{align*}
\Vert M^{-1}\Vert_{\mathrm{spec}} & 
\ll\max_{t,T\leq L}\biggl(\tau_{t}
\prod_{\underset{\ell\neq t}{\ell\leq L}}
(\tau_{\ell}-\tau_{t})\biggr)^{-1}\sum_{\underset{\ell_{1},\ell_{2},\ldots,\ell_{L-T}
\neq\ell}{\ell_{1}<\ell_{2}<\ldots<\ell_{L-T}\leq L}}
\tau_{\ell_{1}}\ldots\tau_{\ell_{L-T}}\\
 & \ll\max_{t\leq L}\biggl(\tau_{t}
 \prod_{\underset{\ell\neq t}{\ell\leq L}}(\tau_{\ell}-\tau_{t})\biggr)^{-1}.
\end{align*}
Notice that $h_\ell \asymp e^{q_\ell \theta}$ and
\[
\left|\tau_{\ell}-\tau_{t}\right|=
\left|\frac{h_{\ell}-h_{t}}{(h_{\ell}-s)(h_{t}-s)}
\right|\gg e^{-\theta(q_\ell + q_t)}\abs{h_{\ell}-h_t}.
\]
Consequently, 
\[
\Vert M^{-1}\Vert_{\mathrm{spec}}\ll    
\max_{t\leq L}\bigg( 
e^{\theta ((L-1)q_t + \sum_{t\neq \ell\leq L}q_\ell)} 
\prod_{\underset{\ell\neq t}{\ell\leq L}}
\abs{h_{\ell}-h_t}^{-1}\bigg)
\]
as required.
\end{proof}

The following lemma is a direct result of van der Corput's lemma with an amplitude function (see for example \cite[Lemma 5.1.4]{Huxley1996}, the details of the proof can be found in \cite[Lemma 3.3]{TechnauYesha2020}

\begin{lemma}[localized van der Corput's lemma]
\label{lem: localised Van der Corput}Let $\mathcal{J}$ be a compact
interval. Let $\varphi:\mathcal{J}\rightarrow\mathbb{R}$ be a smooth
function, let $g$ be a real, differentiable function, and
\[
\mathrm{Van}_{\varphi,L}(s):=\max_{i\leq L}\vert\varphi^{(i)}(s)\vert.
\]
If $\varphi^{(L)}$ has at most $C$ zero on $\mathcal{J}$ and $\lambda>0$
is so that 
\[
\mathrm{Van}_{\varphi,L}(s)\geq\lambda
\]
holds throughout $\mathcal{J}$, then 
\[
\int_{\mathcal{J}}\,g(s) e(\varphi(s))\,\mathrm{d}s\ll V(g)\lambda^{-\frac{1}{L}},
\]
where $V(g)$ is the total variation of $g$ plus the value of $g$ at either endpoint of $\cJ$.
\end{lemma}

\begin{proof}[Proof of Proposition \ref{prop:int s}]
      Combining (\ref{eq: lower bound on norm of vector b}) and Lemma \ref{lem: operator norm bound}
yields 
\[
\mathrm{Van}_{\varphi,L}(s)\gg \max_{t\le L}\left(e^{u_t}e^{\theta((L-1)q_t+ \sum_{t\neq \ell\leq L}q_\ell)}\prod_{\substack{\ell\le L\\\ell \neq t}} \abs{h_{\ell} - h_t}^{-1} \right)  N^{1+2(1-L)\theta}.
\]
The derivatives of the phase function $\varphi$ have a uniformly
bounded number zeros (independent of $\boldsymbol{h}$ and $\boldsymbol{r}$).
Thus Lemma \ref{lem: localised Van der Corput} applies and we infer
\begin{align*}
  I(\vect{h},\vect{r})&\ll\mathfrak{K}_{\vect{u}}(\vect{\mu}_0) \mathfrak{N}_{\vect{q}}\left(\vect{\mu}_0,s\right) \frac{e^{u_1+\dots+u_m}}{(r_1r_2\cdots r_m)^{(1-\Theta)/2}}\max_{t\le L} \left(e^{-u_t}e^{\theta((L-1)q_t+\sum_{t\neq \ell\leq L}q_\ell)}\prod_{\substack{\ell\le L\\\ell \neq t}} \abs{h_{\ell} - h_t}^{-1} \right)^{1/L}.
\end{align*}
\end{proof}

\section{Proof of Lemma \ref{lem:MP = KP non-isolating}}

As demonstrated above, by extracting the various main terms and applying the $B$-process we conclude that
\begin{align*}
  \cK_{m}(N) &= \left(\lim_{N\to \infty} \sum_{\cP\in \mathscr{P}_m} \cM_{\cP}(N) \right)  +  O\left(\frac{1}{N^m} \sum_{\vect{h},\vect{q},\vect{r},\vect{u}}\eta(\vect{r},\vect{h})  \frac{1}{(r_1r_2\cdots r_m)^{(\Theta+1)/2}}\cO_N\right) + o(1)
\end{align*}
as $N \to \infty$. Inserting the bound \eqref{exp varphi}, we deduce
\begin{align*}
  &\Err := \frac{1}{N^m} \sum_{\vect{h},\vect{q},\vect{r},\vect{u}} \eta(\vect{r},\vect{h})  \frac{1}{(r_1r_2\cdots r_m)^{(\Theta+1)/2}}\cO_N\\
  &\ll \frac{1}{N^m} \sum_{\vect{h},\vect{q},\vect{r},\vect{u}}\eta(\vect{r},\vect{h})  \frac{e^{u_1+\dots+u_m}}{r_1r_2\cdots r_m} \mathfrak{K}_{\vect{u}}(\vect{\mu}_0) \mathfrak{N}_{\vect{q}}\left(\vect{\mu}_0,s \right) \max_{t\le m} \left(e^{-u_t}e^{\theta((m-1)q_t+\sum_{t\neq \ell\leq L}q_\ell)}\prod_{\substack{\ell\le m\\\ell \neq t}} \abs{h_{\ell} - h_t}^{-1} \right)^{1/m}.
\end{align*}
Recall that $\mu_0 = r_i \left(\frac{h_i}{\beta \Theta}\right)^{1/(\Theta-1)}$, thus, the condition imposed by $\sum_{\vect{q}}\mathfrak{N}_{\vect{q}}\left(\vect{\mu}_0,s\right) $ implies that $h_i \ll N^{\theta}$. Now we can bound the sum over $\vect{h}$ by using a generalized version of H\"{o}lder's inequality. That is we fix exponents $1/p_1 +1/p_2 \dots  + 1/p_{m-1}=1$. In this case, choose $p_i = m$ for $i \le m-2$ and $p_{m-1}=m/2$ 
\begin{align*}
  \sum_{\vect{h}} \prod_{\substack{\ell\le m\\\ell \neq t}} \abs{h_{\ell} - h_t}^{-1/m}
  &\ll \sum_{t =1}^m \sum_{\substack{h_i\\i \neq t}}\left(\sum_{h_t} \prod_{\substack{\ell\le m\\\ell \neq t}}\abs{h_{\ell} - h_t}^{-1/m}\right)\\
  &\ll m \sum_{\substack{h_i\\i>  1}} \left\{\left(\sum_{h_1} \abs{h_{m-1} - h_1}^{-1/2}\right)^{2/m} \prod_{\ell=2}^{ m-1}  \left(\sum_{h_1}\abs{h_{\ell} - h_1}^{-1}\right)^{1/m}\right\} \\
  &\ll \log(N)^{\frac{m-2}{m}} N^{\theta(m-1)}  N^{\theta/m} \ll N^{\theta((m-1)+1/m)+\varepsilon} .
\end{align*}
Thus
\begin{align*}
  \Err   &\ll N^{\frac{(m^2+m-1)\theta - 1}{m}+\varepsilon}.
\end{align*}
Hence, if $ \theta < 1/(m^2+m-1) $, and $\varepsilon>0$ is taken small enough, then $\Err = o(1)$. From there, the decomposition at the start of Section \ref{s:Triple Correlation} and a standard approximation argument are enough to establish Theorem \ref{thm:main}.
\qed

   \small 
   \section*{Acknowledgements}
   NT was supported by a Schr\"{o}dinger Fellowship of the Austrian Science Fund (FWF): project J 4464-N. We thank Zeev Rudnick for comments on an earlier draft of the paper.

  \bibliographystyle{alpha}
  \bibliography{biblio}

\begin{thebibliography}{EBMV15}

\bibitem[AEBM21]{AistleitnerEl-BazMunsch2021}
C.~Aistleitner, D.~El-Baz, and M.~Munsch.
\newblock A pair correlation problem, and counting lattice points with the zeta
  function.
\newblock {\em to appear in GAFA, see arXiv:2009.08184 [math.NT]}, 2021.

\bibitem[Apo76]{Apostol1976}
T.~Apostol.
\newblock {\em Introduction to analytic number theory}.
\newblock Undergraduate Texts in Mathematics. Springer-Verlag, New
  York-Heidelberg, 1976.

\bibitem[BKY13]{BlomerKhanYoung2013}
V.~Blomer, R.~Khan, and M.~Young.
\newblock Distribution of mass of holomorphic cusp forms.
\newblock {\em Duke Math. J.}, 162(14):2609--2644, 2013.

\bibitem[BZ00]{BocaZaharescu2000}
F.~Boca and A.~Zaharescu.
\newblock Pair correlation of values of rational functions (mod p).
\newblock {\em Duke Math. J.}, 105(2):267--307, 2000.

\bibitem[CY21]{ChaubeyYesha2021}
S.~Chaubey and N.~Yesha.
\newblock The distribution of spacings of real-valued lacunary sequences modulo
  one.
\newblock {\em arXiv preprint arXiv:2108.00431}, 2021.

\bibitem[EBMV15]{El-BazMarklofVinogradov2015a}
D.~El-Baz, J.~Marklof, and I.~Vinogradov.
\newblock The two-point correlation function of the fractional parts of
  {$\sqrt{n}$} is {P}oisson.
\newblock {\em Proc. Amer. Math. Soc.}, 143(7):2815--2828, 2015.

\bibitem[EM04]{ElkiesMcMullen2004}
N.~Elkies and C.~McMullen.
\newblock Gaps in {$\sqrt{n} \operatorname{ mod } 1$} and ergodic theory.
\newblock {\em Duke Math. J.}, 123(1):95--139, 2004.

\bibitem[FKZ21]{FassinaKimZaharescu2021}
M.~Fassina, S.~Kim, and A.~Zaharescu.
\newblock {The Distribution of Spacings Between the Fractional Parts of
  {$n^d\alpha$}}.
\newblock {\em IMRN}, 03 2021.

\bibitem[Hux96]{Huxley1996}
M.~Huxley.
\newblock {\em Area, lattice points, and exponential sums}, volume~13 of {\em
  London Mathematical Society Monographs. New Series}.
\newblock The Clarendon Press, Oxford University Press, New York, 1996.
\newblock Oxford Science Publications.

\bibitem[IK04]{IwaniecKowalski2004}
H.~Iwaniec and E.~Kowalski.
\newblock {\em Analytic number theory}, volume~53 of {\em AMS Colloquium
  Publications}.
\newblock AMS, Providence, RI, 2004.

\bibitem[KL21]{KurlbergLester2021}
P.~Kurlberg and S.~Lester.
\newblock Poisson spacing statistics for lattice points on circles.
\newblock {\em 2112.08522}, 2021.

\bibitem[KR99]{RudnickKurlberg1999}
P.~Kurlberg and Z.~Rudnick.
\newblock The distribution of spacings between quadratic residues.
\newblock {\em Duke Math. J.}, 100, 1999.

\bibitem[LST21]{LutskoSourmelidisTechnau2021}
C.~Lutsko, T.~Sourmelidis, and N~Technau.
\newblock Pair correlation of the fractional parts of {$\alpha n^\theta$}.
\newblock {\em arXiv:2106.09800}, 2021.

\bibitem[Mar07]{Marklof2007}
J.~Marklof.
\newblock Distribution modulo one and {R}atner's theorem.
\newblock In A.~Granville and Z.~Rudnick, editors, {\em Equidistribution in
  Number Theory, An Introduction}, pages 217--244, Dordrecht, 2007. Springer
  Netherlands.

\bibitem[RS96]{RudnickSarnak1996}
Z.~Rudnick and P.~Sarnak.
\newblock {Zeros of principal L-functions and random matrix theory}.
\newblock {\em Duke Mathematical Journal}, 81(2):269--322, 1996.

\bibitem[RS98]{RudnickSarnak1998}
Z.~Rudnick and P.~Sarnak.
\newblock The pair correlation function of fractional parts of polynomials.
\newblock {\em Comm. in Math. Phys.}, 194(1):61--70, 1998.

\bibitem[RSZ01]{RudnickSarnakZaharescu2001}
Z.~Rudnick, P.~Sarnak, and A.~Zaharescu.
\newblock The distribution of spacings between the fractional parts of
  {$n^2\alpha$}.
\newblock {\em Inventiones Mathematicae}, 145(1):37--57, 2001.

\bibitem[RT21]{RudnickTechnau2021}
Z.~Rudnick and N.~Technau.
\newblock The metric theory of the pair correltion function of small
  non-integer powers.
\newblock {\em arXiv:2107.07092}, 2021.

\bibitem[RZ99]{RudnickZaharescu1999}
Z.~Rudnick and A.~Zaharescu.
\newblock A metric result on the pair correlation of fractional parts of
  sequences.
\newblock {\em Acta Arith.}, 89(3):283--293, 1999.

\bibitem[RZ02]{RudnickZaharescu2002}
Z.~Rudnick and A.~Zaharescu.
\newblock The distribution of spacings between fractional parts of lacunary
  sequences.
\newblock {\em Forum Mathematicum}, 14(5):691--712, 2002.

\bibitem[Ste93]{Stein1993}
E.~M. Stein.
\newblock {\em Harmonic analysis: real-variable methods, orthogonality, and
  oscillatory integrals}, volume~3.
\newblock Princeton University Press, 1993.

\bibitem[TY20]{TechnauYesha2020}
N.~Technau and N.~Yesha.
\newblock On the correlations of {$\alpha n^\theta$} mod 1.
\newblock {\em arXiv:2006.16629}, 2020.

\end{thebibliography}

    \hrulefill

    \vspace{4mm}
     \noindent Department of Mathematics, Rutgers University, Hill Center - Busch Campus, 110 Frelinghuysen Road, Piscataway, NJ 08854-8019, USA. \emph{E-mail: \textbf{chris.lutsko@rutgers.edu}}

\vspace{4mm}
     \noindent
Department of Mathematics,
California Institute of Technology,
1200 E California Blvd.,
Pasadena, CA 91125, USA
\emph{E-mail: \textbf{ ntechnau@caltech.edu}}

\end{document}